\documentclass[12pt]{amsart}
\usepackage{amsmath, amssymb, amsthm, bm, float, mathtools, enumitem, caption}
\usepackage{hyperref}

\usepackage[noabbrev,capitalize]{cleveref}
\usepackage{adjustbox}
\crefname{equation}{}{}
\usepackage{ytableau}
\usepackage{graphicx, tikz}
\usepackage[textheight=8.75in, textwidth=6.7in]{geometry}
\usepackage{subcaption}
\usepackage{microtype}

\newtheorem{theorem}{Theorem}[section]
\newtheorem{lemma}[theorem]{Lemma}

\newtheorem*{conjecture*}{Conjecture}
\newtheorem*{tworemarks}{Two Remarks}
\newtheorem*{problem}{Problem}

\theoremstyle{definition}

\theoremstyle{remark}
\newtheorem*{remark}{Remark}

\numberwithin{equation}{section}

\newcommand{\Proportion}{\mathrm{Pr}_2}

\newcommand{\SL}{\mathrm{SL}}

\newcommand{\Q}{\mathbb Q}

\newcommand{\Z}{\mathbb Z}
\newcommand{\ord}{\operatorname{ord}}
\newcommand{\lcm}{\operatorname{lcm}}
\newcommand{\leg}[2]{\left(\frac{#1}{#2}\right)}

\title[A note on Odd partition numbers]{A note on Odd partition numbers}
\dedicatory{Celebrating George E. Andrews and Bruce C. Berndt on their 85th birthdays}
\author{Michael Griffin and Ken Ono}
\address{Dept. of Mathematics, Vanderbilt University, Nashville, TN 37240}
\email{michael.j.griffin@vanderbilt.edu}
\address{Dept. of Mathematics, University of Virginia, Charlottesville, VA 22904}
\email{ko5wk@virginia.edu}

\begin{document}
\begin{abstract}  Ramanujan's celebrated partition congruences modulo $\ell\in \{5, 7, 11\}$ assert that
$$
p(\ell n+\delta_{\ell})\equiv 0\pmod{\ell},
$$
where $0<\delta_{\ell}<\ell$ satisfies $24\delta_{\ell}\equiv 1\pmod{\ell}.$  By proving Subbarao's Conjecture, Radu showed that there are no such congruences when it comes to parity. There are infinitely many odd (resp. even) partition numbers in every arithmetic progression.
For primes $\ell \geq 5,$ we give a new proof of the conclusion that there are infinitely many $m$ for which $p(\ell m+\delta_{\ell})$ is odd. This proof uses a generalization, due to the second author and Ramsey, of a result of Mazur in his classic paper on the Eisenstein ideal. We also refine a classical criterion of Sturm for modular form congruences, which allows us to show that the smallest such $m$ satisfies $m<(\ell^2-1)/24,$
 representing a significant improvement to the previous bound. 
\end{abstract}

\maketitle
\section{Introduction and Statement of Results}

A {\it partition of size $n$} is any nonincreasing sequence of positive integers that sums to $n$. The partition function $p(n)$ counts the partitions of size $n$, and has the convenient generating function
$$
\sum_{n=0}^{\infty} p(n)q^n=\prod_{n=1}^{\infty}\frac{1}{1-q^n}=1+q+2q^2+3q^3+5q^4+7q^5+11q^6+\dots.
$$
Ramanujan famously proved \cite{Ramanujan} the congruences
\begin{displaymath}
\begin{split}
p(5n+4)&\equiv 0\pmod 5,\\
p(7n+5)&\equiv 0\pmod 7,\\
p(11n+6)&\equiv 0\pmod{11},
\end{split}
\end{displaymath}
which are uniformly described by the congruence
\begin{equation}\label{RamaAPS}
p(\ell n+\delta_{\ell})\equiv 0\pmod{\ell},
\end{equation}
where $\ell \in \{5, 7, 11\}$ and $0<\delta_{\ell}<\ell$ satisfies $24\delta_{\ell}\equiv 1\pmod{\ell}.$

Here we consider the parity of $p(n).$ Table~\ref{table} offers some values of
$$
\Proportion(N):= \frac{\#\{ 0\leq n < N \ : \ p(n)\ {\text {\rm is even}}\}}{N},
$$
 the proportion of the first $N$ values that are even.

\begin{table}[H]
\begin{tabular}{|r|r| }
\hline \rule[-3mm]{0mm}{8mm}
$N\ \ \ \ $ &$\Proportion(N)\ \ \ $\\
\hline
200,000&0.5012\dots\\
600,000&0.5000\dots\\
1,000,000&0.5004\dots\\
$\infty\ \ \ \  $ &$\frac{1}{2}?\ \ \ \ \ $\\
\hline
\end{tabular}

\caption{Proportion of even values}{\label{table}}
\end{table}

These numerics suggest the widely held belief that the parity of the partition function is randomly distributed \cite{ParkinShanks}. Unfortunately, very little is known. Perhaps the strongest result in this direction is Radu's proof \cite{Radu} of Subbarao's Conjecture, which asserts that every arithmetic progression $r\pmod t$ contains infinitely many integers $N\equiv r\pmod t$ for which $p(N)$ is even, and infinitely many integers $M\equiv r\pmod t$ for which $p(M)$ is odd.  Radu's work built on previous papers \cite{OnoSubbarao1, OnoSubbarao2} by the second author, which proved the ``even cases'' of the conjecture, and offered partial results in the ``odd cases''. To be precise, the second author proved that there are infinitely many $M\equiv r\pmod t$ for which $p(M)$ is odd, provided that there is at least one such $M.$ Moreover, if there is such an $M,$ then he proved that the smallest one satisfies
\begin{equation}\label{bound}
M<\frac{2^{23+j}\cdot 3^7t^6}{d^2}\prod_{\substack{p\mid 6t\\ prime}}\left(1-\frac{1}{p^2}\right)-2^j,
\end{equation}
where $d:=\gcd(24r-1,t)$ and $j$ is an integer for which $2^j>t/24.$

\begin{problem} Determine an upper bound for the smallest $N\equiv r\pmod t$ for which $p(N)$ is even.
\end{problem}

We offer a new proof of the odd case of Subbarao's Conjecture for the  family of arithmetic progressions including those in (\ref{RamaAPS}). Moreover, in these cases we obtain a significant improvement to the bound in (\ref{bound}).

\begin{theorem}\label{MainThm}
If $\ell \geq 5$ is prime, then there are infinitely many $m$ for which $p(\ell m+\delta_{\ell})$ is odd. Moreover, the smallest such $m$ satisfies $m<(\ell^2-1)/24.$
\end{theorem}
\begin{tworemarks} \ \newline
\noindent
(1) In a way, Theorem~\ref{MainThm} is ``sharp''. Indeed, for $\ell=5,$ we find that $p(4)$, which happens to equal 5, must be odd, as $m=0<(5^2-1)/24=1$ is the only option. 

\noindent
(2) Theorem~\ref{MainThm} has a generalization where $\ell$ can be replaced with  any $t\in \Z_+$ coprime to 6. The smallest $m$ for which $p(tm+\delta_t)$ is odd, where $0<\delta_t<t$ and $24\delta_t\equiv 1\pmod t$, satisfies 
$$
m< \frac{t}{24}  \prod_{\substack{p\mid t\\ prime}}p.
$$
\end{tworemarks}

To prove this theorem, we apply a result (see Theorem~\ref{OnoRamseyThm}) by the second author and Ramsey \cite{OnoRamsey} that is a generalization of a result of Mazur in his work on the Eisenstein ideal. We also require a refinement of a classical criterion of Sturm (see Theorem~\ref{SturmTheorem}) for modular form congruences. In Section~\ref{MainProof} we apply these results to specific modular functions obtained by applying Hecke operators to a weight 0 eta-quotient that encodes the parity of $p(n).$

\section*{Acknowledgements}
\noindent
 The second author is grateful for the support of the Thomas Jefferson Fund and the NSF
(DMS-2002265 and DMS-2055118). 

\section{Preliminaries about modular form congruences}

The proof of Theorem~\ref{MainThm} makes use of  the theory of modular forms (for example, see \cite{OnoCBMS}). To be precise, we require a ``modulo $\ell$'' Atkin-Lehner theorem, and a refinement of a classical criterion of Sturm for modular form congruences. Here we recall and derive these results.

\subsection{A modulo $\ell$ Atkin-Lehner theorem}

The main result of 
\cite{OnoRamsey} is a ``modulo'' $\ell$ Atkin-Lehner Theorem\footnote{As we let $\ell$ denote the primes in (\ref{RamaAPS}), we state Theorem 2.1 with $p$ in place of $\ell$ to avoid confusion.}, a generalization of a result of Mazur (suggested by Serre) in his classical paper on the Eisenstein ideal (see p. 83 of \cite{Mazur}).

\begin{theorem}[Theorem 1.1 of \cite{OnoRamsey}]\label{OnoRamseyThm}
Suppose that $f(\tau)$ is a meromorphic modular form of weight $k\in \frac{1}{2}\Z$ on $\Gamma_0(N)$ which has integral Fourier coefficients at $i\infty.$ Let $p$ be a prime with the property that there is an integer $c>1$ for which the Fourier expansion of $f(\tau)$ at $i\infty$ (i.e. $q:=e^{2\pi i \tau}$ throughout) satisfies
$$
f(\tau)\equiv \sum_{n\geq n_0}a(n)q^{cn}\pmod{p}.
$$
If $\gcd(Np,c)=1$ and $f$ has trivial nebentypus character at $p$ (i.e. $f$ is fixed by the diamond operators at $p$) if $p\mid N,$ then we have that
$$
f(\tau)\equiv a(0) \pmod{p}.
$$
\end{theorem}

\subsection{Refinement of Sturm's Criterion}

We also require a refinement of the well-known theorem of Sturm (see Theorem 1 of \cite{Sturm}) that offers a criterion for proving modular form congruences. Namely, we give its straightforward refinement to infinite dimensional spaces of {\it weakly holomorphic modular forms}, those whose poles (if any) are supported at cusps.

We consider integer weight $k$ modular forms $g(\tau),$ and we suppose that $T>1$ is an integer for which
$$g(\tau)=\sum a(n) q^n \in \Z\left(\left(q^{1/T}\right)\right).
$$
We let $\ord_q(g)$ denote the least $n \in \frac{1}{T}\Z$ for which $a(n)\neq 0,$ and for primes $p$ we let $\ord_p(g)$ be the least $n$ for which $a(n) \not \equiv 0 \pmod{p}.$  Our refinement of Sturm's theorem takes into account the orders at all the cusps of a congruence subgroup, not just the single cusp at $i\infty.$ We recall that the order of $g$ at a cusp $a/c$ is given by 
\[\ord_q\left (g|_k\left(\begin{smallmatrix}a&b\\c&d\end{smallmatrix}\right)\right),\]
where $\left(\begin{smallmatrix}a&b\\c&d\end{smallmatrix}\right)\in \SL_2(\Z)$ and
\[g|_k \gamma:= (c\tau+d)^{-k}\det(\gamma)^{k/2}
g\left( \frac {a\tau+b}{c\tau+d}\right).\]

\begin{theorem}\label{SturmTheorem}
Suppose that $f(\tau)$ is a non-zero weakly holomorphic modular form of weight $k\in \frac{1}{2}\Z$ on a congruence subgroup $\Gamma$ with multiplier $\nu$. If $f(\tau)$ has integral coefficients at $i\infty$ and
\begin{equation}\label{SturmB}
\ord_{p}(f) > \frac{k}{12}[\SL_2(\Z):\Gamma]-\sum _{\substack{[\gamma]\in \Gamma\backslash \SL_2(\Z)\\ [\gamma]\neq [ I ]}} \ord_{ q}\left(  f|_k \gamma\right),
\end{equation}
then $f\equiv 0 \pmod{p}.$ 
\end{theorem}
\begin{remark}
Sturm's theorem for holomorphic modular forms follows as $\ord_{ q}\left(  f|_k \gamma\right)\geq 0$.  
\end{remark}
\begin{proof}
The proof follows almost exactly as the proof of Sturm's theorem. 
For each $\gamma$ in $\SL_2(\Z)$, we have  that the coefficients of $f|_k \gamma$ lie in the cyclotomic field $K= \Q(e^{\frac{2\pi i}{\lcm(N,M)}}),$ where $N$ is the level of $\Gamma$ and $M$ is the order of the multiplier (for example, see Thm. 6.6 of \cite{Shimura}).
Let $\pi$ be any prime ideal of $K$ which divides $p$. We clear denominators and factor out extra divisibility if necessary. To be precise, the Chinese Remainder Theorem guarantees that there is some $A_{[\gamma]} \in K$ for which $A_{[\gamma]} f|_k \gamma$ has $K$-integral coefficients, but also has the property that the coefficients do not all lie in $\pi$. 

We choose a positive integer $m$ for which $\nu^m$ is trivial and $12\mid mk$. This has the result of multiplying each term in (\ref{SturmB}) by $m.$ 
Using this $m$, we consider the modular norm 
\[
F(\tau)=\prod _{[\gamma]\in \Gamma\backslash \SL_2(\Z)} A^m_{[\gamma]} f^m|_k\gamma,
\]
where we take $A_{[I]}=1.$ 
Obviously, $F$ is a weight $k\cdot [\SL_2(\Z):\Gamma]$ weakly holomorphic modular form on $\SL_2(\Z).$ Moreover, we find that 
\begin{eqnarray*}\label{ordF}
\ord_\pi(F)&=&\sum _{[\gamma]\in \Gamma\backslash \SL_2(\Z)} \ord_{ \pi}\left( A^m_{[\gamma]} f^m|_k \gamma\right)
 \geq  \ m\cdot \ord_\pi(f)+m\cdot \sum _{\substack{[\gamma]\in \Gamma\backslash \SL_2(\Z)\\ [\gamma]\neq [I]}} \ord_q\left(  f|_k \gamma\right).
\end{eqnarray*}
Using  $\Delta(\tau)=q-24q^2+\dots,$ the unique normalized weight 12 cusp form on $\SL_2(\Z),$ which is nonvanishing on the upper-half of the complex plane, we find that
 $F/\Delta^{\frac{mk}{12}\cdot[\SL_2(\Z):\Gamma]}$ is a weight $0$ weakly holomorphic modular function on $\SL_2(\Z).$ Therefore, it  is an integral polynomial in the $j$-function
 $$
 j(\tau):=\frac{E_4(\tau)^3}{\Delta(\tau)}=q^{-1}+744+196684q+\dots\in \Z((q)).
 $$
However, if (\ref{SturmB}) holds, then the coefficients of this polynomial must be divisible by $\pi$, which means that $F\equiv 0\pmod{\pi}$. By construction, each $\ord_{ \pi}\left( A^m_{[\gamma]} f^m|_k \gamma\right)$ is finite for $[\gamma]\neq [I]$. Therefore, the only possibility is that $\ord_{ \pi}\left( f^m\right)= m\cdot \ord_p(f)$ is infinite. 
\end{proof}

\section{Proof of Theorem~\ref{MainThm}}\label{MainProof}

To prove Theorem~\ref{MainThm}, we employ Theorems~\ref{OnoRamseyThm} and \ref{SturmTheorem}. We apply these results to modular functions obtained by applying the Hecke operators to a distinguished modular function that encodes the parity of $p(n).$ Namely, in terms of Dedekind's eta-function $\eta(\tau):=q^{1/24}\prod_{n=1}^{\infty}(1-q^n),$ we consider 
\begin{equation}
G(\tau)=\sum_{n\in \Z} a(n)q^{\frac{n}{3}}:=\frac{\eta(\tau)^8}{\eta(2\tau)^8}=q^{-\frac{1}{3}}-8q^{\frac{2}{3}}+28q^{\frac{5}{3}}-64q^{\frac{8}{3}}+\dots.
\end{equation}
For primes $\ell \geq 5,$ we consider the functions 
\begin{equation}
G_{\ell}(\tau):=\left(G|U_{\ell}\right)(\tau)=\frac{1}{\ell}\sum_{j\mod \ell} G\left(\frac{\tau+3j}{\ell}\right)=
\sum_{n\in \Z} a(\ell n)q^{\frac{n}{3}}.
\end{equation}
We use $3j$ in the summation, as opposed to $j$, to maintain integrality of the Fourier coefficients.

\begin{lemma}\label{GModularFunction}
The following are true for the eta-quotient $G(\tau):=\eta(\tau)^8/\eta(2\tau)^8.$\smallskip

\noindent
(1) We have that $G(\tau)$ is a modular function on $\Gamma_0(2),$ with a non-trivial multiplier $\nu$ that takes values in the third roots of unity.

\noindent
(2) The multiplier $\nu$ is trivial on $\Gamma_0(2)\cap X(3)$, where 
\[X(3) = \left\{\gamma \in \SL_2(\Z) \ : \ \gamma\equiv \pm I  \text{ or } \gamma \equiv \pm \begin{pmatrix}0&-1\\1&0\end{pmatrix} \pmod{3}\right\}.\] 

\noindent
(3)  If $\ell\geq 5$ is prime, then $G_{\ell}(\tau):=G|U_{\ell}$ is a modular function on $\Gamma_0(2\ell)$ with multiplier $\nu^{\ell}.$
\end{lemma}
\begin{proof}
These claims follow from standard facts about eta-quotients (for example,  see
\cite{Newman}).
\end{proof}

To prove Theorem~\ref{MainThm}, we require the behavior of $G_{\ell}(\tau)$ at the four inequivalent cusps of $\Gamma_0(2\ell),$ which can be taken to be $\{0, 1/6, 1/3\ell, i\infty\}$.
 Thus, in order to compute the order of $G_\ell$ at each cusp, it suffices to consider $G_\ell|_0 \gamma$ for 
\begin{equation}\label{cusps}
\gamma\in \left\{
\begin{pmatrix}0&-1\\1&0\end{pmatrix}, 
\begin{pmatrix}1&0\\6&1\end{pmatrix},
\begin{pmatrix}1&0\\3\ell&1\end{pmatrix}, I
\right\}.
\end{equation}
We have chosen the second and third representatives to have the lower-left entry divisible by 3 to simplify the contribution of the multiplier $\nu.$
The next lemma gives the order of $G_{\ell}(\tau)$ at these cusps.

\begin{lemma}\label{CuspOrders} If $\ell \geq 5$ is prime, then the following are true.

\noindent
(1) If we have $\gamma\in \left\{
\left(\begin{smallmatrix}0&-1\\1&0\end{smallmatrix}\right),
\left(\begin{smallmatrix}1&0\\6&1\end{smallmatrix}\right),
\left(\begin{smallmatrix}1&0\\3\ell&1\end{smallmatrix}\right), I
\right\},$ then
\[
\ord_q(G_\ell |_0\gamma) \begin{cases}
= \frac{1}{6\ell} & \text{if } \ \gamma=\left(\begin{smallmatrix}0&-1\\1&~0\end{smallmatrix}\right),\\
=-\frac{\ell}{3} & \text{if } \ \gamma=\left(\begin{smallmatrix}1&0\\6&1\end{smallmatrix}\right),\\
\geq \frac{1}{6}& \text{if } \ \gamma=\left(\begin{smallmatrix}1&0\\3\ell&1\end{smallmatrix}\right),\\
 \geq \frac{1}{3}& \text{if } \ \gamma=I.
\end{cases}\]

\noindent
(2) If $\gamma=\left(\begin{smallmatrix}a&b\\c&d\end{smallmatrix}\right)\in \SL_2(\Z),$ then we have that
\[
\ord_q(G_\ell |_0\gamma) \begin{cases}
 = \frac{1}{6\ell}& \text{if } \ \gcd(c,2\ell)=1,\\
 =-\frac{\ell}{3} & \text{if } \ \gcd(c,2\ell)=2,\\
\geq \frac{1}{6}& \text{if } \ \gcd(c,2\ell)=\ell,\\
\geq \frac{1}{3} & \text{if } \ \gcd(c,2\ell)=2\ell.
\end{cases}\]
\end{lemma}

\begin{proof}
The proof of (1) follows  case-by-case. \smallskip

\noindent
{\bf Case $\gamma=\left(\begin{smallmatrix}1&0\\0&1\end{smallmatrix}\right)$}: In this case we are simply bounding the order $\ord_q(G_\ell).$
The exponents of $G$ are all positive and in $\Z-1/3.$ Therefore, $G |U_\ell $ has exponents which are positive and in $\Z-\frac{1}{3}\leg{\ell}{3},$ where $\leg{\ell}{3}$ is the Legendre symbol. Therefore, we have 
\[\ord_q\left(G_\ell\right)\geq 1/3.\]
 
\noindent
{\bf Case $\gamma=\left(\begin{smallmatrix}1&0\\6&1\end{smallmatrix}\right)$}:
By direct calculation, we find that
\begin{eqnarray*}
G_\ell|_0 \begin{pmatrix}1&0\\6&1\end{pmatrix} &=& \frac1\ell G |_0 \sum_{j \mod{\ell}} \begin{pmatrix}1&3j\\0&\ell\end{pmatrix}\begin{pmatrix}1&0\\6&1\end{pmatrix}
=\frac1\ell G|_0 \sum_{j \mod{\ell}} \begin{pmatrix}1+18j&3j\\6\ell&\ell\end{pmatrix}.
\end{eqnarray*}
Here we are summing the $\ell$ images of $\frac{1}{\ell}G$ under $|_0.$
For each $j$ with $18j\not \equiv -1 \pmod{\ell}$,  we find integers $A_j$ and $B_j$ so that $-6\ell A_j +(1+18j)B_j=1$. Additionally, we may choose $A_j$ and $B_j$ so that $3\mid A_j$ and $B_j\equiv 1\pmod{18}$. One then finds that $\left(\begin{smallmatrix}1+18j&3j\\6\ell&\ell\end{smallmatrix}\right)=\left(\begin{smallmatrix}1+18j&A_j\\6\ell&B_j\end{smallmatrix}\right)\left(\begin{smallmatrix}1&(1-B_j)/6)\\ 0&\ell\end{smallmatrix}\right),$ where the matrix $\left(\begin{smallmatrix}1+18j&A_j\\6\ell&B_j\end{smallmatrix}\right)$ is in $\Gamma_0(2)\cap X(3)$ and so acts trivially on $G$. 

On the other hand, if $18j \equiv -1 \pmod{\ell}$,  we have that  $\left(\begin{smallmatrix}1+18j&3j\\6\ell&\ell\end{smallmatrix}\right)=\left(\begin{smallmatrix}\frac{1+18j}{\ell}&3j\\6&\ell\end{smallmatrix}\right)\left(\begin{smallmatrix}\ell&0\\0&1\end{smallmatrix}\right),$ where $\left(\begin{smallmatrix}\frac{1+18j}{\ell}&3j\\6&\ell\end{smallmatrix}\right)\in \Gamma_0(2)\cap X(3),$ and so acts trivially on $G$.
Therefore,  we find that
\begin{eqnarray*}
G_\ell|_0 \begin{pmatrix}1&0\\6&1\end{pmatrix} &=&\frac1\ell G|_0\sum_{j \mod \ell} \begin{pmatrix}1+18j&3j\\6\ell&\ell\end{pmatrix}\\
&=& \frac1\ell G|_0\left[\sum_{\substack{j \mod \ell\\ 18j\neq -1}} 
\begin{pmatrix}1&(1-B_j)/6)\\ 0&\ell\end{pmatrix}+ \begin{pmatrix}\ell&0\\0&1\end{pmatrix}\right]\\
&=& G|_0U_\ell -\frac1\ell G|_0\begin{pmatrix}1&3 j' \\ 0&\ell\end{pmatrix} +\frac{1}{\ell} G|_0V_\ell,
\end{eqnarray*}
where in the last equation $j'$  satisfies $ 18 j'\equiv 1\pmod{\ell}.$ In particular, we find that 
\[\ord_q\left(G_\ell|_0 \begin{pmatrix}1&0\\6&1\end{pmatrix} \right)=\ord_q\left(G|V_\ell\right) = -\ell/3.\]

\noindent
{\bf Case $\gamma=\left(\begin{smallmatrix}1&0\\3\ell&1\end{smallmatrix}\right)$:}
 We follow a similar calculation to the one above. For each $j$ modulo ${\ell}$ there are integers $A_j$ and $B_j$ so that $-3\ell^2 A_j+B_j(1+9j\ell)=1$. Moreover, we can choose $A_j$ and $B_j$ so that $3\mid A_j$, $2\mid B_j$, and $B_j\equiv 1 \pmod{18\ell}$. Then we have the following:

\begin{eqnarray*}
G_\ell|_0 \begin{pmatrix}1&0\\3\ell&1\end{pmatrix} &=& \frac1\ell G|_0 \sum_{j \mod{\ell}} \begin{pmatrix}1&3j\\0&\ell\end{pmatrix}\begin{pmatrix}1&0\\3\ell&1\end{pmatrix}\\
&=&\frac1\ell G|_0 \sum_{j \mod{\ell}} \begin{pmatrix}1+9j\ell&3j\\3\ell^2&\ell\end{pmatrix}\\
&=&\frac1\ell G|_0 \sum_{j \mod{\ell}} 
\begin{pmatrix}1+9j\ell&A_j\\3\ell^2&B_j\end{pmatrix}
\begin{pmatrix}1&\frac{1-B_j}{3\ell}\\0&\ell\end{pmatrix}\\
&=&\frac1\ell G|_0\sum_{j \mod{\ell}} 
\begin{pmatrix}-A_j&1+9j\ell\\-B_j&3\ell^2\end{pmatrix}
\begin{pmatrix}0&-1\\1&0\end{pmatrix}
\begin{pmatrix}1&\frac{1-B_j}{3\ell}\\0&\ell\end{pmatrix}\\
&=& G|_0 \begin{pmatrix}0&-1\\1&0\end{pmatrix} |U_\ell
\end{eqnarray*}
\noindent
For the last step, we note that  the $\left(\begin{smallmatrix}-A_j&1+9j\ell\\-B_j&3\ell^2\end{smallmatrix}\right)$ are in $\Gamma_0(2)\cap X(3),$ and so act trivially on $G.$

Using that $G=\eta(\tau)^8/\eta(2\tau)^8$ and the transformation law for $\eta$, we find that 
\[G|_0 \begin{pmatrix}0&-1\\1&0\end{pmatrix} = 2^4 \frac{\eta(\tau)^8}{\eta(\tau/2)^8}=\frac{2^4}{G(\tau/2)}.\]
The exponents of $G|_0\left(\begin{smallmatrix}0&-1\\1&0\end{smallmatrix}\right) $ are all positive and in $1/2\Z+1/6.$ Therefore, $G|_0 \left(\begin{smallmatrix}0&-1\\1&0\end{smallmatrix}\right) |U_\ell $ has exponents which are positive and in $\frac{1}{3}\Z+\frac{1}{6}\leg{\ell}{3},$ where $\leg{\ell}{3}$ is the Legendre symbol. Therefore, we have 
\[\ord_q\left(G_\ell|_0 \begin{pmatrix}1&0\\3\ell&1\end{pmatrix}\right)\geq 1/6.\]

\noindent
{\bf Case $\gamma=\left(\begin{smallmatrix}0&-1\\1&~0\end{smallmatrix}\right)$:}
Arguing as above, we find that
\begin{eqnarray*}
G_\ell|_0 \begin{pmatrix}0&-1\\1&~0\end{pmatrix}&=& 
G|_0\begin{pmatrix}0&-1\\1&~0\end{pmatrix}| \left[ U_\ell - \frac{1}{\ell}\begin{pmatrix}1&0\\0&\ell\end{pmatrix}+\frac{1}{\ell} V_\ell\right].
\end{eqnarray*}
Therefore, we have that
\[\ord_q\left(G_\ell|_0 \begin{pmatrix}0&-1\\1&0\end{pmatrix}\right) = \ord_q\left(G|_0 \begin{pmatrix}0&-1\\1&0\end{pmatrix}\begin{pmatrix}1&0\\0&\ell\end{pmatrix}\right) = \frac{1}{6\ell}.\]

Finally, to prove (2), we note that the set of cosets $\Gamma_0(2\ell)\backslash \SL_2(\Z)$ has $3(\ell+1)$ elements, which may be partitioned into four subsets corresponding to  the inequivalent cusps considered in (1). These subsets are determined by  $\gcd(c,2\ell)$, where $c$ is the lower-left entry of any representative of a coset.  We note that the number of cosets in the subset is equal to $2\ell / \gcd(c,2\ell).$
Most importantly, we have that if $\gamma_1$ and $\gamma_2$ correspond to the same cusp, then for any weight $k$ modular form on $\Gamma_0(2\ell)$, we have that 
$\ord_q\left(f|_k\gamma_1\right)=\ord_q\left(f|_k\gamma_2\right), $ even if $f|_k\gamma_1\neq f|_k\gamma_2$. 
\end{proof}

\subsection{Proof of Theorem~\ref{MainThm}}
We consider
$F(\tau):=G(3\tau)= \eta^8(3\tau)/\eta(6\tau)^8,$ which satisfies
\[F (\tau)= q^{-1}\prod_{n=1}^{\infty}\frac{(1-q^{3n})^8}{(1-q^{6n})^8}\equiv q^{-1}\prod_{n=1}^{\infty}\frac{1}{(1-q^{24n})} \equiv \sum_{n=0}^\infty p(n)q^{24n-1} \pmod{2}.\]
Moreover, we have that $F(\tau)$ is a weakly holomorphic modular form of weight $0$ (i.e. modular function) on $\Gamma_0(18)$ with a trivial nebentypus character (See Theorem 1.64 of  \cite{OnoCBMS}). 

For positive integers $n$ coprime to 6, the weight 0 Hecke operators $T_n$ preserve the spaces of even weight weakly holomorphic modular forms on $\Gamma_0(18)$ (for example, see Proposition 2.3 of \cite{OnoCBMS}). If $\ell\geq 5$ is prime, then we have
$$
 F |_0 T_{\ell} := F| U_{\ell} + \frac{1}{\ell} F|V_{\ell}\equiv F| U_{\ell}+F|V_{\ell}\pmod 2.
 $$
If $F|U_\ell$ vanishes modulo $2$, then we would have
$$F|_0T_\ell \equiv F|V_\ell \equiv q^{-\ell}+\sum_{n=1}^{\infty} p(n)q^{\ell(24n-1)}\pmod 2.
$$
Theorem \ref{OnoRamseyThm}  implies that $F|_0T_\ell$ is congruent to a constant modulo $2$.  This is obviously false as is seen by the presence of the $q^{-\ell}$ term. Hence, $F|U_\ell$ cannot vanish modulo $2$, which guarantees the existence of infinitely many $m$ for which $p(\ell m+\delta_{\ell})$ is odd by Theorem 2 of \cite{OnoSubbarao1, OnoSubbarao2}.

To complete the proof, we now bound the smallest such $m,$ which is equivalent to bounding
the exponent of the first odd term of $F|U_\ell$. 
As $G|U_\ell =G_{\ell} \not \equiv 0 \pmod{2},$ Theorem~\ref{SturmTheorem} implies that
\begin{eqnarray*}
\ord_2(G_\ell) &\leq&  \sum _{\substack{[\gamma]\in \Gamma\backslash \SL_2(\Z)\\ [\gamma]\neq [ I ]}} \ord_{ q}\left(  G_\ell|_0\gamma\right)
\leq  -\ell\cdot \left(\frac{-\ell}{3}\right) - 2 \cdot\left(\frac{1}{6}\right) -2\ell \cdot \left(\frac{1}{6\ell}\right)= \frac{\ell^2}{3} -\frac{2}{3}.
\end{eqnarray*}
The three terms in the middle expression correspond to the cusps inequivalent to infinity, weighted by their $\SL_2(\Z)$ multiplicity.
We are using the fact that the number of cosets in the subset corresponding to a cusp for $\gamma=\left(\begin{smallmatrix} a&b\\c&d\end{smallmatrix}\right)$  is $2\ell / \gcd(c,2\ell).$

Furthermore, $\ord_2(G_\ell)=(24M-1)/3\ell$ where $M$ is the smallest positive integer such that $\ell \mid (24M-1)$ and $p(M)$ is odd. Changing variables with $M=\ell m+\delta_\ell$, the bound on $\ord_2(G_\ell)$ becomes
$$
\frac{24\ell m+24\delta_\ell -1}{3\ell} \leq\frac{\ell^2-2}{3},
$$
In terms of $m$, this implies that
$$
m\leq \frac{1}{24}\left( \ell^2-2-\frac{24\delta_\ell-1}{\ell} \right)< \frac{\ell^2-1}{24}.
$$
This completes the proof.

\end{document}